\documentclass[11pt]{article}
\usepackage{amssymb}
\usepackage{amsfonts}
\usepackage{amsmath}
\usepackage[shortlabels]{enumitem}
\usepackage{float}
\usepackage[caption = false]{subfig}
\usepackage{epsfig}
\usepackage{graphicx,epstopdf}

\newtheorem{theorem}{Theorem}[section]

\newtheorem{lemma}[theorem]{Lemma}

\newtheorem{remark}{Remark}[section]

\newcommand{\func}[1]{\operatorname{#1}}
\newenvironment{proof}[1][Proof]{\noindent\textbf{#1.} }{\ \rule{0.5em}{0.5em}}
\numberwithin{equation}{section}

\begin{document}

\title{Integro-differential equations linked to compound birth processes
with infinitely divisible addends}
\author{Luisa Beghin\thanks{%
Sapienza University of Rome, P.le A. Moro 5, 00185 Roma, Italy. e-mail:
\texttt{luisa.beghin@uniroma1.it}} \and Janusz Gajda\thanks{%
University of Warsaw, Poland} \and Aditya Maheshwari\thanks{%
Indian Institute of Management Indore, India} }
\date{}
\maketitle

\begin{abstract}
Stochastic modelling of fatigue (and other material's
deterioration), as well as of cumulative damage in risk theory,
are often based on compound sums of independent random variables,
where the number of addends is represented by an independent
counting process. We consider here a cumulative model where,
instead of a renewal process (as in the Poisson case), a linear
birth (or Yule) process is used. This corresponds to the
assumption that the frequency of \textquotedblleft damage"
increments accelerates according to the increasing number of
\textquotedblleft damages". We start from the partial differential
equation satisfied by its transition density, in the case of
exponentially distributed addends, and then we generalize it by
introducing a space-derivative of convolution type (i.e. defined
in terms of the Laplace exponent of a subordinator). Then we are
concerned with the solution of integro-differential equations,
which, in particular cases, reduce to fractional ones.
Correspondingly, we analyze the related cumulative jump processes
under a general infinitely divisible distribution of the
(positive) jumps. Some special cases (such as the stable, tempered
stable, gamma and Poisson) are presented.\\


\end{abstract}

\section{Introduction}

Compound counting processes, especially in the Poisson case, are widely
studied and applied in many different fields: in reliability theory (for
studying the development of fatigue in materials) and in collective risk
theory, where many cumulative damage models are defined in terms of the
following sum%
\begin{equation}
Y(t):=\sum_{j=0}^{N(t)}X_{j}.  \label{1}
\end{equation}%
The addends $X_{j}$ are assumed to be i.i.d. random variables, for $%
j=1,2,... $, and $N:=\{N(t),t\geq 0\}$ is an independent counting (i.e. a
non-negative, integer valued and non-decreasing) stochastic process. In the
special case where $N$ is a homogeneous Poisson process, with rate $\lambda
, $ the well-known compound Poisson process is obtained.

Under the assumption that the counting process $N$ is Poisson with parameter
$\lambda >0,$ and the jumps are exponentially distributed with parameter $%
\xi $, it has been proved in \cite{BEMA} that the distribution of the
compound Poisson process defined in (\ref{1}), can be written in terms of a
Wright function. Moreover the density of its absolutely continuous component
$f_{Y}(y,t),$ for $y,t>0$ satisfy the following differential equation
\begin{equation}
\xi \frac{\partial }{\partial t}f=-\left[ \lambda +\frac{\partial }{\partial
t}\right] \frac{\partial }{\partial y}f,  \label{comp5}
\end{equation}%
with conditions%
\begin{equation}
\left\{
\begin{array}{l}
f(y,0)=0 \\
\int_{0}^{+\infty }f(y,t)dy=1-e^{-\lambda t}%
\end{array}%
\right. .  \label{die1}
\end{equation}%
This process has been generalized in \cite{BEMA} to the fractional case and
in \cite{LUD} the assumption of Gamma distributed jumps has been considered.

We consider here the cumulative process (\ref{1}), when the number of
addends, instead of being represented by a renewal counting process (as in
the Poisson case), is assumed to be a linear birth (or Yule) process $%
B:=\{B(t),t\geq 0\}$, with one progenitor and with rate $\lambda
_{k}=k\lambda $, for $k=1,2,...,$ i.e.%
\begin{equation}
Y(t)=\sum_{j=1}^{B(t)}X_{j},\qquad t\geq 0,  \label{bi}
\end{equation}%
under the assumption that $X_{j}$ are i.i.d. positive random variables,
independent from $B.$ This model can be also described by assuming that each
member of a population gives birth independently to one offspring at an
exponential time with rate $\lambda $. Moreover each member of the
population produces a random (positive) \textquotedblleft damage" that
contribute individually to the \textquotedblleft total damage" of the
population. Alternatively, we can assume that $X_{j}$'s represent the claim
size of the $j$-th policy holder and that the number of claims for the
insurance company evolves in time according to a birth process $B(t)$, i.e.
when, for example, the arrival rate of the new claims is proportional to the
number of the claims previously arrived. In this case a crucial random
variable is represented by the time to ruin of the company, i.e.%
\begin{equation*}
\tau :=\inf \{t:Y(t)>ct+u\},\qquad u,c>0,
\end{equation*}%
where $u$ is the initial capital and $c$ is constant risk premium rate. For
applications to the risk theory of compound birth processes, see the very
recent paper \cite{MIN}.

This kind of process has been introduced by \cite{SOB} for modelling crack
growth with accelerating frequency of increments as the crack grows. Also in
the cumulative models applied to reliability theory, the study of the
first-passage time through a certain critical value, i.e. $T_{\beta }=\inf
\{t\geq 0:Y(t)>\beta \}$, is crucial. In particular, $T_{\beta }$ is the
fatigue failure time, in the cumulative fatigue model. Since, in case of
non-negative addends, $Y(t)$ is non-decreasing for any $t$, we get%
\begin{equation}
P\{T_{\beta }>t\}=P\{Y(t)<\beta \}=F_{Y}(\beta ,t).  \label{rel}
\end{equation}

We recall that the birth process $B$ is a continuous-time Markov process
and, in the linear case, its probability mass function is given by

\begin{equation*}
p_{n}(t):=P\left\{ \left. B(t)=n\right\vert B(0)=1\right\} =e^{-\lambda
t}\left( 1-e^{-\lambda t}\right) ^{n-1},\qquad t\geq 0,\;n=1,2,...
\end{equation*}%
In the fractional case, the birth processes have been studied in \cite{ORPO}%
, \cite{ORPO3} and \cite{CAPO}. In \cite{LAN} a non-markovian generalization
of the Yule process has been introduced.

We start by deriving the partial differential equation satisfied by the
transition density of (\ref{bi}) in the case of exponentially distributed
addends; then we generalize it by introducing a space-derivative $\mathcal{D}%
_{y}^{g}$ of convolution-type (defined by means of the Laplace exponent $%
g(\theta )$ of a subordinator). More precisely, we will be concerned with
the the solution to the integro-differential equation%
\begin{equation}
\frac{\partial }{\partial t}f(y,t)=-\lambda \frac{e^{\lambda t}}{\xi }%
\mathcal{D}_{y}^{g}\left[ f\ast f\right] (y,t),\qquad y,t,\lambda ,\xi >0,
\label{oo}
\end{equation}%
under certain initial conditions, where we denote by $f_1 \ast f_2$ the
convolution of the functions $f_1$ and $f_2.$ For the (integral) definition
of $\mathcal{D}_{y}^{g}$ see (\ref{toa}) below.

We will prove that the solution of (\ref{oo}) coincides with the density $%
f_{Y_{g}}$ of (\ref{bi}) when the distribution of the addends is
extended (from the exponential case) inside the class of
infinitely divisible laws. Note that, for
$g(\theta)=\theta^{\alpha}$ and $\alpha \in (0,1)$, the derivative
$\mathcal{D}_{y}^{g}$ reduces to a Caputo fractional derivative of
order $\alpha$ (see (\ref{toa}) below) and then equation
(\ref{oo}) becomes a fractional differential equation. Some
special cases (such as the stable, tempered stable, gamma and
Poisson cases) will be illustrated. As we will see, our model will
prove to be more flexible and adaptable to the real data provided
that the appropriate distributions of the addends and the
corresponding parameters' values are chosen.

\

Let $g:(0,+\infty )\rightarrow \mathbb{R}$ be a Bernstein function, i.e. let
$g$ be non-negative, infinitely differentiable and such that, for any $x\in
(0,+\infty ),$%
\begin{equation*}
(-1)^{n}\frac{d^{n}}{dx^{n}}g(x)\leq 0,\qquad \text{for any }n\in \mathbb{N}%
_{.}
\end{equation*}%
A function $g$ is a Bernstein function if and only if it admits the
following representation%
\begin{equation*}
g(x)=a+bx+\int_{0}^{+\infty }(1-e^{-sx})\overline{\nu }(ds),
\end{equation*}%
for $a,b\in \mathbb{R},$ where $\overline{\nu }$ is the corresponding L\'{e}%
vy measure and $\left( a,b,\overline{\nu }\right) $ is called the L\'{e}vy
triplet of $g$. Then a subordinator is the stochastic process with
non-decreasing paths $\mathcal{A}_{g}:=\left\{ \mathcal{A}_{g}(t),t\geq
0\right\} ,$ such that%
\begin{equation}
\mathbb{E}e^{-\theta \mathcal{A}_{g}(t)}=e^{-g(\theta )t},  \label{sub}
\end{equation}%
i.e. $g(\theta )$ is the Laplace exponent of $\mathcal{A}_{g}.$ Let moreover
$\mathcal{L}_{g}(t),$ $t\geq 0,$ be its inverse, i.e.
\begin{equation*}
\mathcal{L}_{g}(t)=\inf \left\{ s\geq 0:\mathcal{A}_{g}(s)>t\right\} ,\qquad
t>0
\end{equation*}%
and $l_{g}(x,t)=\Pr \left\{ \mathcal{L}_{g}(t)\in dx\right\} /dx$ be its
transition density.

We recall the definition of the convolution-type derivative on the positive
half-axes, in the sense of Caputo (see \cite{TOA}, Def.2.4, for $b=0$) :%
\begin{equation}
\mathcal{D}_{t}^{g}u(t):=\int_{0}^{t}\frac{d}{ds}u(t-s)\nu (s)ds,\qquad t>0,
\label{toa}
\end{equation}%
where $\nu $ is the tail of the L\'{e}vy measure $\overline{\nu },$ i.e. $%
\nu (s)=\int_{s}^{+\infty }\overline{\nu }(dz)$. Convolution-type
derivatives (or derivatives defined as integrals with memory kernels) have
been treated recently by many authors: see, among the others, \cite{KOC},
\cite{GAJ}, \cite{TOA2}.

The Laplace transform of $\mathcal{D}_{t}^{g}$ is given by
\begin{equation}
\int_{0}^{+\infty }e^{-\theta t}\mathcal{D}_{t}^{g}u(t)dt=g(\theta )%
\widetilde{u}(\theta )-\frac{g(\theta )}{\theta }u(0),\qquad \mathcal{R}%
(\theta )>\theta _{0},  \label{lapconv}
\end{equation}%
(see \cite{TOA}, Lemma 2.5). It is easy to check that, in the trivial case
where $g(\theta )=\theta $, the convolution-type derivative coincides with
the first-order derivative, while, for $g(\theta )=\theta ^{\alpha },$ for $%
\alpha \in (0,1)$, it coincides with the Caputo fractional derivative (see
e.g. \cite{KIL}, p.90) of order $\alpha .$

\section{The exponential case}

As a preliminary result, we consider the case of the compound birth process $%
Y$ with exponentially distributed addends: see \cite{CRO} and \cite{SOB} for
possible applications, in survival analysis and reliability theory,
respectively.

\begin{lemma}
Let $X_{j}$ be i.i.d. $Exp(\xi )$, for $j=1,2,...$, then the density of $Y,$
defined in (\ref{bi}), i.e.%
\begin{equation}
f_{Y}(y,t)=\xi \exp \left\{ -\lambda t-\xi e^{-\lambda t}y\right\} ,\qquad
y,t\geq 0,  \label{den}
\end{equation}%
satisfies the equation%
\begin{equation}
\frac{\partial }{\partial t}f=-\lambda \frac{\partial }{\partial y}\left(
yf\right) ,\qquad y,t\geq 0,  \label{eq}
\end{equation}%
with $f(y,0)=f_{X}(y)=\xi e^{-\xi y}.$
\end{lemma}

\begin{proof}
By a conditioning argument, we can write%
\begin{equation*}
f_{Y}(y,t)=\sum_{n=1}^{\infty }p_{n}(t)f_{X}^{\ast (n)}(y),
\end{equation*}%
which, under Laplace transform, gives%
\begin{eqnarray}
\widetilde{f}_{Y}(\theta ,t) &=&\sum_{n=1}^{\infty }p_{n}(t)\widetilde{f}%
_{X}^{\ast (n)}(\theta )=\sum_{n=1}^{\infty }p_{n}(t)\left[ \widetilde{f}%
_{X}(\theta )\right] ^{n}  \label{den2} \\
&=&e^{-\lambda t}\sum_{n=1}^{\infty }\left( 1-e^{-\lambda t}\right) ^{n-1}
\left[ \frac{\xi }{\xi +\theta }\right] ^{n}=\frac{\xi e^{-\lambda t}}{%
\theta +\xi e^{-\lambda t}}.  \notag
\end{eqnarray}%
It can be easily checked that (\ref{den}) satisfies equation (\ref{eq}) and
the initial condition.
\end{proof}

In Fig.1 we plot the probability density function (hereafter pdf) of the
process $Y$, defined in (\ref{bi}) (in the case of exponentially distributed
addends), estimated directly from the realizations of the process. Moreover
we compare it with the theoretical pdf given in (\ref{den}), for different
values of $t$. One can notice the perfect agreement between the theoretical
and empirical pdf's.
\begin{figure}[th]
\centering
\includegraphics[width=.5\textwidth, height
=.3\textwidth]{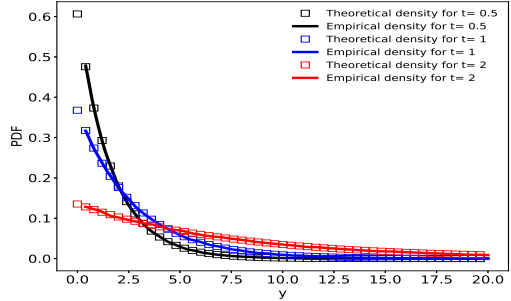}
\caption{Compound birth process $Y$, with exponential distributed addends:
theoretical and empirical pdf's for different values of $t$ and for $\protect%
\lambda =\protect\xi =1$.}
\label{figure1}
\end{figure}

\begin{remark}
Equation (\ref{eq}) can be considered as special case of\ the Fokker-Planck
equation with null diffusion coefficient. Moreover, it can be alternatively
written as%
\begin{equation}
\frac{\partial }{\partial t}f=-\lambda \frac{e^{\lambda t}}{\xi }\frac{%
\partial }{\partial y}\left( f\ast f\right) ,\qquad y,t\geq 0.  \label{gen}
\end{equation}%
Indeed, we have that%
\begin{eqnarray*}
\frac{e^{\lambda t}}{\xi }\frac{\partial }{\partial y}\left( f_{Y}\ast
f_{Y}\right) &=&\frac{e^{\lambda t}}{\xi }\xi ^{2}e^{-2\lambda t}\frac{%
\partial }{\partial y}\int_{0}^{y}\exp \left\{ -\xi e^{-\lambda t}z\right\}
\exp \left\{ -\xi e^{-\lambda t}(y-z)\right\} dz \\
&=&\xi e^{-\lambda t}\frac{\partial }{\partial y}\left( y\exp \left\{ -\xi
e^{-\lambda t}y\right\} \right) =\frac{\partial }{\partial y}\left(
yf_{Y}\right) .
\end{eqnarray*}%
In the next section, we will generalize the governing equation of our model
in the form given in (\ref{gen}).
\end{remark}

The distribution of the first-passage time of $Y$ through the level $\beta
>1 $ can be written as follows:
\begin{equation}
P\{T_{\beta }<t\}=\left\{
\begin{array}{l}
0,\qquad t\leq 0 \\
\exp \left\{ -\xi e^{-\lambda t}\beta \right\} ,\qquad t>0%
\end{array}%
\right.  \label{df}
\end{equation}%
by considering (\ref{rel}). The jump in zero of (\ref{df}) is equal to $%
P\{T_{\beta }=0\}=P\{X>\beta \}=e^{-\xi \beta }$. The absolutely continuous
component thus displays a Gumbel-type distribution. Its mean value can be
obtained as follows:%
\begin{eqnarray}
\mathbb{E}T_{\beta } &=&\int_{0}^{+\infty }\left[ 1-\exp \left\{ -\xi
e^{-\lambda t}\beta \right\} \right] dt=-\sum_{j=1}^{\infty }\frac{(-\xi
\beta )^{j}}{j!}\int_{0}^{+\infty }e^{-\lambda tj}dt  \label{et} \\
&=&-\frac{1}{\lambda }\sum_{j=1}^{\infty }\frac{(-\xi \beta )^{j}}{%
j^{2}(j-1)!}=\frac{\xi \beta }{\lambda }\sum_{l=0}^{\infty }\frac{(-\xi
\beta )^{l}}{l!}\frac{\Gamma (l+1)^{2}}{\Gamma (l+2)^{2}}  \notag \\
&=&\frac{\xi \beta }{\lambda }\,_{2}\Psi _{2}\left[ \left. -\xi \beta
\right\vert
\begin{array}{cc}
(1,1) & (1,1) \\
(2,1) & (2,1)%
\end{array}%
\right]  \notag
\end{eqnarray}%
where
\begin{equation*}
_{p}\Psi _{q}\left[ \left. x\right\vert
\begin{array}{c}
(a_{l},\alpha _{l})_{1,p} \\
(b_{l},\beta _{l})_{1,q}%
\end{array}%
\right] \text{,\qquad }x,a_{l},b_{j}\in \mathbb{C},\;\alpha _{l},\beta
_{j}\in \mathbb{R},\;l=1,...,p,\;j=1,...,q
\end{equation*}%
for $p,q\in \mathbb{N},$ is the Fox-Wright function (see \cite{KIL}, p.56).
The asymptotic behavior of $\mathbb{E}T_{\beta }$ can be studied by
considering formula (1.12.68), p.67 in \cite{KIL}, so that we can write%
\begin{equation*}
\mathbb{E}T_{\beta }=\frac{\xi \beta }{\lambda }\,H_{2,3}^{1,2}\left[ \left.
\xi \beta \right\vert
\begin{array}{cc}
(0,1) & (0,1) \\
(0,1) & (-1,1)%
\end{array}%
\begin{array}{c}
\, \\
(-1,1)%
\end{array}%
\right] ,
\end{equation*}%
where
\begin{equation*}
H_{p,q}^{m,n}\left[ \left. x\right\vert
\begin{array}{c}
(a_{p},A_{p}) \\
(b_{q},B_{q})%
\end{array}%
\right] \text{,\qquad }x,a_{i},b_{j}\in \mathbb{C},\;A_{l},B_{j}\in \mathbb{R%
}^{+},\;i=1,...,p,\;j=1,...,q
\end{equation*}%
for $p,q,m,n\in \mathbb{N}$, with $0\leq n\leq p$, $1\leq m\leq q$, is the
H-function (see e.g. \cite{MAT}, p.2). Then, by applying Theorem 1.2, p.19
in \cite{MAT}, for $\mu =\alpha =1>0$ and $d=\min \left\{ -1,-1\right\} =-1$
we get, for $\beta \rightarrow +\infty ,$%
\begin{equation*}
\mathbb{E}T_{\beta }\simeq \frac{\xi \beta }{\lambda }O(\beta ^{-1})=\frac{%
\xi }{\lambda }O(1).
\end{equation*}

\section{The infinitely divisible cases}

We now extend the previous results, by considering the relaxation equation
with the convolution-type derivative defined in (\ref{toa}). Indeed, it is
well-known that the survival (or reliability) function $\Phi
_{X}(x):=\int_{x}^{+\infty }f_{X}(u)du$ of the r.v. $X\sim Exp(\xi )$
satisfies the so-called relaxation equation%
\begin{equation}
\frac{d}{dx}u(x)=-\xi u(x),\qquad x\geq 0,  \label{rela}
\end{equation}%
with initial condition $u(0)=1.$ As we will see later, replacing the
space-derivative with $\mathcal{D}_{x}^{g}$ corresponds to generalize the
distribution of $X$, inside the class of infinitely divisible r.v.'s.

\begin{lemma}
Let $\mathcal{A}_{g}$ denote the subordinator defined by (\ref{sub}), then
the initial-value problem%
\begin{equation}
\left\{
\begin{array}{l}
\mathcal{D}_{x}^{g}u(x)=-\xi u(x) \\
u(0)=1,%
\end{array}%
\right.  \label{in}
\end{equation}%
with $x\geq 0$, $\xi >0,$ is satisfied by%
\begin{equation}
u(x)=\xi \int_{0}^{+\infty }e^{-\xi t}P\left\{ \mathcal{A}_{g}(t)\geq
x\right\} dt.  \label{pr}
\end{equation}
\end{lemma}

\begin{proof}
We take the Laplace transform of (\ref{in}), by considering (\ref{lapconv}),
and get, for $\mathcal{R}(\theta )>\theta _{0},$%
\begin{equation*}
g(\theta )\widetilde{u}(\theta )-\frac{g(\theta )}{\theta }u(0)=-\xi
\widetilde{u}(\theta ),
\end{equation*}%
so that%
\begin{equation}
\widetilde{u}(\theta )=\frac{g(\theta )}{\theta }\frac{1}{g(\theta )+\xi }.
\label{rr}
\end{equation}%
On the other hand, from (\ref{pr}) we get%
\begin{eqnarray*}
\widetilde{u}(\theta ) &=&\xi \int_{0}^{+\infty }e^{-\theta
x}dx\int_{0}^{+\infty }e^{-\xi t}dt\int_{x}^{+\infty }h_{g}(y,t)dy=\xi
\int_{0}^{+\infty }e^{-\xi t}dt\int_{0}^{+\infty
}h_{g}(y,t)dy\int_{0}^{y}e^{-\theta x}dx \\
&=&\frac{\xi }{\theta }\int_{0}^{+\infty }e^{-\xi t}dt\int_{0}^{+\infty
}h_{g}(y,t)(1-e^{-\theta y})dy=\frac{\xi }{\theta }\left[ \frac{1}{\xi }-%
\frac{1}{\xi +g(\theta )}\right] ,
\end{eqnarray*}%
which is equal to (\ref{rr}). The initial condition is clearly satisfied by (%
\ref{pr}).
\end{proof}

In the trivial case $g(\theta )=\theta $, the differential equation in (\ref%
{in}) reduces to (\ref{rela}), which is satisfied by $\Phi _{X}(x)=e^{-\xi
x}.$ For $g(\theta )=\theta ^{\alpha },$ $\alpha \in (0,1),$ equation (\ref%
{in}) coincides with the well-known fractional relaxation equation (see e.g.
\cite{GLO}, \cite{MAI} and \cite{BEG}). In all the other cases, the solution
corresponds to the survival function of a subordinator stopped at an
independent exponential time, i.e. $\mathcal{A}_{g}(X).$

Alternatively, (\ref{pr}) gives the probability that the subordinator $%
\mathcal{A}_{g}$ hits (or crosses) a certain level in a random time smaller
than an exponentially distributed one. Indeed, by considering that
\begin{equation*}
P\left\{ \mathcal{A}_{g}(t)\geq x\right\} =P\left\{ \mathcal{L}_{g}(x)\leq
t\right\} ,
\end{equation*}%
we can write the solution of (\ref{in}) as
\begin{equation}
u(x)=\xi \int_{0}^{+\infty }e^{-\xi t}P\left\{ \mathcal{L}_{g}(x)\leq
t\right\} dt=P\left\{ \mathcal{L}_{g}(x)\leq X\right\} ,  \label{pl}
\end{equation}%
with $X\sim Exp(\xi ).$ In the special case where $g(\theta )=\sqrt{\theta }$%
, the inverse stable subordinator $\mathcal{L}_{g}$ is equal in distribution
to a Brownian motion reflecting in the origin and thus (\ref{pl}) reduces to
the probability that the Brownian motion is under an exponentially
distributed barrier at time $x$, i,e, $P\left\{ |B(x)|\leq X\right\} $ (see
\cite{BEG}).

Let us now consider the r.v.'s $X_{j}^{g}$, for $j=1,2,...$, with survival
function $\Phi _{X^{(g)}}$ coinciding with (\ref{pr}). They are thus
non-negative, infinitely divisible, with the following Laplace exponent%
\begin{equation}
\psi _{X^{(g)}}(\theta ):=-\log \mathbb{E}e^{-\theta X^{(g)}}=\log \left( 1+%
\frac{g(\theta )}{\xi }\right) ,\qquad \xi ,\theta >0.  \label{be}
\end{equation}%
Clearly, in the special case $g(\theta )=\theta ,$ formula (\ref{be})
reduces to the Laplace exponent of the exponential law. Thus we assume here
that the density of the addends in (\ref{1}) is given by%
\begin{equation}
f_{X^{(g)}}(x)=\xi \int_{0}^{+\infty }e^{-\xi t}h_{g}(x,t)dt,\qquad \xi ,y>0,
\label{lo}
\end{equation}%
(where $h_{g}(\cdot ,t)$ is the density of the subordinator $\mathcal{A}_{g}$%
), or, alternatively, that the following equality in law holds : $X^{(g)}%
\overset{d}{=}\mathcal{A}_{g}(X)$.

\begin{theorem}
The solution to equation%
\begin{equation}
\frac{\partial }{\partial t}f(y,t)=-\lambda \frac{e^{\lambda t}}{\xi }%
\mathcal{D}_{y}^{g}\left[ f\ast f\right] (y,t),\qquad y,t>0,  \label{eq2}
\end{equation}%
under the initial condition $f(y,0)=f_{X^{(g)}}(y),$ is given by the density
function of the process
\begin{equation}
Y_{g}(t)=\sum_{j=1}^{B(t)}X_{j}^{(g)},\qquad t\geq 0.  \label{ef}
\end{equation}

\begin{proof}
We take the Laplace transform of (\ref{eq2}), w.r.t. $y$, so that we get, by
considering (\ref{lapconv}):%
\begin{equation}
\frac{\partial }{\partial t}\widetilde{f}(\theta ,t)=-\lambda \frac{%
e^{\lambda t}}{\xi }g(\theta )\left[ \widetilde{f}(\theta ,t)\right] ^{2},
\label{eq3}
\end{equation}%
with $\widetilde{f}(\theta ,0)=\frac{\xi }{\xi +g(\theta )}.$ We can check,
by differentiating, that the solution to (\ref{eq3}) is equal to%
\begin{equation}
\widetilde{f}(\theta ,t)=\frac{\xi e^{-\lambda t}}{\xi e^{-\lambda
t}+g(\theta )}.  \label{fy}
\end{equation}%
On the other hand we can start from the distribution function of (\ref{ef})
which can be written as
\begin{equation}
F_{Y_{g}}(y,t):=P\{Y_{g}(t)<y\}=\sum_{n=1}^{\infty }P(B(t)=n)F_{X_{g}}^{\ast
(n)}(y),  \label{conv}
\end{equation}%
where $F_{X_{g}}^{\ast (n)}$ denotes the $n$-th convolution of $%
F_{X^{(g)}}(x):=P\{X^{(g)}<x\}$. Under the assumption of absolutely
continuous and positive random addends $X_{j}^{(g)},$ for $j=1,2,..$, with
density $f_{X^{(g)}}(x):=P\{X^{(g)}\in dx\}/dx,$ we can write%
\begin{equation*}
F_{X_{g}}^{\ast (n)}(x)=\int_{0}^{+\infty }F_{X_{g}}^{\ast
(n-1)}(x-z)f_{X}(z)dz.
\end{equation*}%
By denoting $\widetilde{g}(\theta ):=$ $\int_{0}^{+\infty }e^{-\theta
x}g(x)dx$ the Laplace transform of $g:\mathbb{R}^{+}\rightarrow \mathbb{R},$
we get
\begin{equation}
\widetilde{F}_{X_{g}}^{\ast (n)}(\theta )=\left[ \widetilde{F}%
_{X_{g}}(\theta )\right] ^{n}=\frac{\left[ \widetilde{f}_{X_{g}}(\theta )%
\right] ^{n}}{\theta }.  \label{lconv}
\end{equation}%
Therefore we have that%
\begin{eqnarray}
\widetilde{f}_{Y_{g}}(\theta ,t) &=&\sum_{n=1}^{\infty }e^{-\lambda t}\left(
1-e^{-\lambda t}\right) ^{n-1}\left[ \frac{\xi }{\xi +g(\theta )}\right] ^{n}
\label{lt} \\
&=&\frac{\xi e^{-\lambda t}}{\xi e^{-\lambda t}+g(\theta )},  \notag
\end{eqnarray}%
which coincides with (\ref{fy}).
\end{proof}
\end{theorem}

As far as the first-passage time $T_{\beta }$ of the process $Y_{g}$ through
the level $\beta >1$ is concerned, we can derive the general formula of the
Laplace transform of its distribution function$,$ by considering (\ref{rel})
and by taking into account (\ref{conv}) and (\ref{lconv}), as follows:%
\begin{eqnarray}
\mathcal{L}\left[ P\{T_{\beta }<t\};\theta \right] &=&\mathcal{L}\left[
P\{Y(t)>\beta \};\theta \right] =\frac{1}{\theta }-\frac{1}{\theta }%
\widetilde{f}_{Y_{g}}(\theta ,t)  \label{III} \\
&=&\frac{g(\theta )}{\theta \left[ e^{-\lambda t}\xi +g(\theta )\right] }.
\notag
\end{eqnarray}

\begin{remark}
The following time-changed representation of the process $Y_{g}$ can be
checked, by proving the equality of the one-dimensional distributions:%
\begin{equation}
Y_{g}(t)\overset{d}{=}\mathcal{A}_{g}(Y(t)),  \label{id}
\end{equation}%
for any $t\geq 0,$ where $Y$ is the compound birth process (with exponential
jumps) defined in (\ref{bi}) and supposed independent of the subordinator $%
\mathcal{A}_{g}.$ Indeed, the Laplace transform of $\mathcal{A}_{g}(Y(t))$
can be written, for any $t\geq 0,$ as
\begin{eqnarray*}
\mathbb{E}e^{-\theta \mathcal{A}_{g}(Y(t))} &=&\int_{0}^{+\infty }\mathbb{E}%
e^{-\theta \mathcal{A}_{g}(z)}f_{Y}(z,t)dz \\
&=&\int_{0}^{+\infty }e^{-g(\theta )z}f_{Y}(z,t)dz \\
&=&[\text{by (\ref{den2})}] \\
&=&\frac{\xi e^{-\lambda t}}{g(\theta )+\xi e^{-\lambda t}},
\end{eqnarray*}%
which coincides with (\ref{lt}). Thus (\ref{id}) follows from the unicity of
the Laplace transform.
\end{remark}

\subsection{Some special cases}

\textbf{(i) The Mittag-Leffler case}

Let $g(\theta )=\theta ^{\alpha }$, for $\alpha \in (0,1],$ then the law of
the addends is Mittag-Leffler of parameters $\alpha ,\xi $ and, in this
special case, we can obtain an explicit and simple formula for the density
of $Y_{g}:=Y_{\alpha }.$ Recall that the Mittag-Leffler function, with two
parameters, is defined as%
\begin{equation}
E_{\beta ,\gamma }(x)=\sum_{j=0}^{\infty }\frac{x^{\beta j}}{\Gamma (\beta
j+\gamma )},\qquad x,\beta ,\gamma \in \mathbb{C},\text{ }\func{Re}\left(
\beta \right) ,\func{Re}\left( \gamma \right) >0.  \label{ml}
\end{equation}%
Then the density of $X^{g}$ can be obtained by applying the well-known
formula of the Laplace transform of the Mittag-Leffler function (see \cite%
{KIL}, formula (1.9.13), for $\rho =1$), i.e.%
\begin{equation}
\mathcal{L}\left\{ x^{\gamma -1}E_{\beta ,\gamma }(Ax^{\beta });s\right\} =%
\frac{s^{\beta -\gamma }}{s^{\beta }-A},  \label{lt2}
\end{equation}%
with $\func{Re}\left( \beta \right) ,\func{Re}\left( \gamma \right) >0,$ $%
A\in \mathbb{R}$ and $s>|A|^{1/\func{Re}\left( \beta \right) }.$ Indeed we
have%
\begin{equation}
f_{X^{(\alpha )}}(x)=\mathcal{L}^{-1}\left\{ \frac{\xi }{\xi +\theta
^{\alpha }};x\right\} =\xi x^{\alpha -1}E_{\alpha ,\alpha }(-\xi x^{\alpha
}),\qquad x,\xi >0,\;\alpha \in (0,1],  \label{rv}
\end{equation}%
which is the density of $\mathcal{A}_{\alpha }(X),$ where $\mathcal{A}%
_{\alpha }$ is the $\alpha $-stable subordinator and $X$ an independent
exponential r.v. The equation (\ref{eq2}), in this case, reduces to the
following space-fractional differential equation of order $\alpha ,$%
\begin{equation}
\frac{\partial }{\partial t}f_{Y_{\alpha }}(y,t)=-\lambda \frac{\partial
^{\alpha }}{\partial y^{\alpha }}\left[ f_{Y_{\alpha }}\ast f_{Y_{\alpha }}%
\right] (y,t),\qquad y,t>0,  \label{sf}
\end{equation}%
with initial condition $f_{Y_{\alpha }}(y,0)=\xi x^{\alpha -1}E_{\alpha
,\alpha }(-\xi y^{\alpha }).$

The derivative appearing in (\ref{sf}) is the Caputo fractional derivative
defined as follows: let $\alpha >0$, $m=\left\lfloor \alpha \right\rfloor +1$
and assume that $u:[a,b]\rightarrow \mathbb{R},$ $b>a,$\ is an absolutely
continuous function, with absolutely continuous derivatives up to order $m$
on $[a,b]$, then, for $x\in \lbrack a,b],$
\begin{equation}
\frac{d^{\alpha }}{dx^{\alpha }}u(x):=\left\{
\begin{array}{l}
\frac{1}{\Gamma (m-\alpha )}\int_{a}^{x}\frac{1}{(x-s)^{\alpha -m+1}}\frac{%
d^{m}}{ds^{m}}u(s)ds,\qquad \alpha \notin \mathbb{N}_{0} \\
\frac{d^{m}}{dx^{m}},\qquad \alpha =m\in \mathbb{N}_{0}%
\end{array}%
\right.  \label{ca}
\end{equation}%
is the Caputo fractional derivative of order $\alpha $ (see \cite{KIL},
p.92). Indeed, for $g(\theta )=\theta ^{\alpha },$ we have that%
\begin{equation}
\mathcal{D}_{x}^{g}=\frac{d^{\alpha }}{dx^{\alpha }},  \label{der}
\end{equation}%
as it can be checked by considering (\ref{toa}) with the L\'{e}vy measure $%
\overline{\nu }(ds)=\alpha s^{-\alpha -1}ds/\Gamma (1-\alpha )$ and the tail
L\'{e}vy measure $\nu (ds)=s^{-\alpha }ds/\Gamma (1-\alpha )$ (see Remark
2.6 in \cite{TOA} for details).

The density of the process (\ref{ef}), which we denote now as $Y_{\alpha },$
is given by
\begin{equation}
f_{Y_{\alpha }}(y,t)=\xi e^{-\lambda t}y^{\alpha -1}E_{\alpha ,\alpha }(-\xi
e^{-\lambda t}y^{\alpha }),\qquad y,t\geq 0,\text{ }\alpha \in (0,1],
\label{tr}
\end{equation}%
as can be obtained by inverting the Laplace transform (\ref{lt}), by means
of (\ref{lt2}), with $g(\theta )=\theta ^{\alpha }$ and $\beta =\gamma
=\alpha .$

The Mittag-Leffler r.v. has infinite moments and thus the same holds for the
process $Y_{\alpha }.$ This is confirmed by the representation $\mathcal{A}%
_{\alpha }(Z)$ of the random addends, since it is well-known that the stable
law has infinite moments of order greater than $\alpha $.

The distribution of the first-passage time through the level $\beta >1$ in
this case can be obtained, by considering (\ref{III}), which reduces to%
\begin{equation}
\mathcal{L}\left[ P\{T_{\beta }<t\};\theta \right] =\frac{\theta ^{\alpha -1}%
}{\theta ^{\alpha }+\xi e^{-\lambda t}}.  \label{lf2}
\end{equation}%
By inverting (\ref{lf2}) we get
\begin{equation}
P\{T_{\beta }<t\}=\left\{
\begin{array}{l}
0,\qquad t\leq 0 \\
E_{\alpha ,1}(-\xi e^{-\lambda t}\beta ^{\alpha }),\qquad t>0%
\end{array}%
\right. ,  \label{lf3}
\end{equation}%
which coincides with (\ref{df}) for $\alpha =1.$ The jump in the origin of (%
\ref{lf3}) coincides with $P\{T_{\beta }=0\}=E_{\alpha ,1}(-\xi \beta
^{\alpha })=P(X^{(\alpha )}>\beta ).$

Moreover, the previous probability converges to zero, as $\beta $ tends to
infinity, but with a power law (instead of exponentially), as can be checked
by recalling the well-known asymptotic behavior of the Mittag-Leffler
function (see \cite{KIL}, formula (1.8.11)), for $|z|\rightarrow +\infty $,
i.e.%
\begin{equation}
E_{\beta ,\gamma }(z)=-\sum_{k=1}^{n}\frac{z^{-k}}{\Gamma (\gamma -\beta k)}%
+O\left( z^{-n-1}\right) ,\quad n\in \mathbb{N},  \label{asy}
\end{equation}%
where $0<\beta <2,$ $\mu \leq \arg (z)\leq \pi $, $\pi \beta /2<\mu <\min
\{\pi ,\pi \beta \}.$ Thus, for $\beta \rightarrow +\infty ,$ we get $%
P\{T_{\beta }<t\}\simeq O(\beta ^{-\alpha }).$

The density of the absolutely continuous component is easily obtained by
taking the first derivative of (\ref{lf3}) and reads%
\begin{equation*}
f_{T_{\beta }}(t)=\frac{\lambda \xi e^{-\lambda t}\beta ^{\alpha }}{\alpha }%
E_{\alpha ,\alpha }(-\xi e^{-\lambda t}\beta ^{\alpha }),\qquad t>0.
\end{equation*}%
We thus obtain the definition of the following fractional extension of the
reflected Gumbel density%
\begin{equation}
f(x)=\frac{e^{-x}}{\alpha E_{\alpha ,\alpha +1}(-\beta ^{\alpha })}E_{\alpha
,\alpha }(-e^{-x}\beta ^{\alpha }),\qquad x>0,  \label{gu}
\end{equation}%
where the normalizing constant is obtained as follows%
\begin{equation*}
\int_{0}^{+\infty }e^{-x}E_{\alpha ,\alpha }(-e^{-x}\beta ^{\alpha
})dx=\int_{0}^{1}E_{\alpha ,\alpha }(-z\beta ^{\alpha })dz=\alpha E_{\alpha
,\alpha +1}(-\beta ^{\alpha }).
\end{equation*}%
The expected first-passage time through $\beta $ can be obtained from (\ref%
{lf3}), as follows:%
\begin{eqnarray*}
\mathbb{E}T_{\beta } &=&-\int_{0}^{+\infty }\sum_{j=1}^{\infty }\frac{(-\xi
e^{-\lambda t}\beta ^{\alpha })^{j}}{\Gamma (\alpha j+1)}dt=\frac{\xi \beta
^{\alpha }}{\lambda }\sum_{l=0}^{\infty }\frac{(-\xi \beta ^{\alpha
})^{l}\left( \Gamma (l+1)\right) ^{2}}{l!\Gamma (\alpha l+\alpha +1)\Gamma
(l+2)} \\
&=&\frac{\xi \beta ^{\alpha }}{\lambda }\,_{2}\Psi _{2}\left[ \left. -\xi
\beta ^{\alpha }\right\vert
\begin{array}{cc}
(1,1) & (1,1) \\
(\alpha +1,\alpha ) & (2,1)%
\end{array}%
\right]
\end{eqnarray*}%
which reduces to (\ref{et}) for $\alpha =1.$ Its asymptotic behavior can be
obtained by applying Theorem 1.2, p.19 in \cite{MAT}, for $\mu =\alpha >0$
and $d=\min \left\{ -1,-1\right\} =-1$ and $\alpha >0$, we get, for $\beta
\rightarrow +\infty ,$%
\begin{eqnarray*}
\mathbb{E}T_{\beta } &=&\frac{\xi \beta ^{\alpha }}{\lambda }\,H_{2,3}^{1,2}%
\left[ \left. \xi \beta ^{\alpha }\right\vert
\begin{array}{cc}
(0,1) & (0,1) \\
(0,1) & (-\alpha ,\alpha )%
\end{array}%
\begin{array}{c}
\, \\
(-1,1)%
\end{array}%
\right] \\
&\simeq &\frac{\xi \beta ^{\alpha }}{\lambda }\,O(\left( \beta ^{\alpha
}\right) ^{-1})=\frac{\xi }{\lambda }\,O(1).
\end{eqnarray*}%
Thus the expected first-passage time through the level $\beta $ converges to
the same limit of the exponential case, for any $\alpha $, even though the
expected value of $Y_{\alpha }$ is infinite, while it is finite for $\alpha
=1$.

\textbf{(ii) The tempered case}

We consider now the case $g(\theta )=(\mu +\theta )^{\alpha }-\mu ^{\alpha }$%
, for $\mu >0,$ which is the Bernstein function of the tempered $\alpha $%
-stable subordinator $\mathcal{A}_{\alpha ,\mu }$, for $\alpha \in (0,1].$
The law of the addends can be written explicitly, by the well-known
relationship between the density of the tempered stable $h_{\alpha ,\mu
}(x,t)$ and that of the stable itself, i.e.%
\begin{equation*}
h_{\alpha ,\mu }(x,t)=\exp \{-\mu x-\mu ^{\alpha }t\}h_{\alpha }(x,t),\qquad
x,t\geq 0.
\end{equation*}%
Since, in this case,
\begin{equation*}
Y_{g}(0)\overset{d}{=}\mathcal{A}_{\alpha ,\mu }(X),
\end{equation*}%
where $X$ is again exponentially distributed with parameter $\xi $ and
independent of $\mathcal{A}_{\alpha ,\mu }$, we have that%
\begin{equation}
f_{X_{\mu }^{(\alpha )}}(x)=\xi e^{-\mu x}\int_{0}^{+\infty }e^{-\xi t-\mu
^{\alpha }t}h_{\alpha }(x,t)dt=\xi e^{-\mu x}x^{\alpha -1}E_{\alpha ,\alpha
}(-(\xi -\mu ^{\alpha })x^{\alpha }),  \label{mu}
\end{equation}%
which generalizes (\ref{rv}), for $\mu \neq 1.$ Analogously to the previous
case, we can write the transition density of the process $Y_{g}:=Y_{\alpha
,\mu }$ as%
\begin{equation}
f_{Y_{\alpha ,\mu }}(y,t)=\xi e^{-\lambda t-\mu y}y^{\alpha -1}E_{\alpha
,\alpha }(-(\xi e^{-\lambda t}-\mu ^{\alpha })y^{\alpha }),\qquad y,t\geq 0,%
\text{ }\alpha \in (0,1].  \label{3.27}
\end{equation}%
Moreover, from Theorem 3.2, we get that (\ref{3.27}) satisfies equation (\ref%
{sf}), where the fractional $\alpha $-order Caputo derivative must be
replaced by the following Caputo-type tempered derivative%
\begin{equation*}
\frac{d^{\alpha ,\mu }}{dx^{\alpha ,\mu }}u(x):=\frac{\alpha \mu ^{\alpha }}{%
\Gamma (1-\alpha )}\int_{0}^{x}\Gamma (-\alpha ;\mu s)\frac{d}{ds}%
u(s)ds,\qquad \alpha \in (0,1),\text{ }\mu >0,
\end{equation*}%
(where $\Gamma (\eta ,x):=\int_{x}^{+\infty }e^{-t}t^{\eta -1}dt$ is the
upper incomplete Gamma function). Indeed the tail L\'{e}vy measure reads $%
\nu (ds)=\frac{\alpha \mu ^{\alpha }\Gamma (-\alpha ;\mu s)ds}{\Gamma
(1-\alpha )}$ (see \cite{TOA}). In this case, the Laplace transform of the
first-passage time distribution (\ref{III}) reduces to%
\begin{equation}
\mathcal{L}\left[ P\{T_{\beta }<t\};\theta \right] =\frac{(\mu +\theta
)^{\alpha }-\mu ^{\alpha }}{\theta \left[ e^{-\lambda t}\xi +(\mu +\theta
)^{\alpha }-\mu ^{\alpha }\right] }.  \label{ll}
\end{equation}%
By inverting (\ref{ll}) and denoting by $\gamma
(a;x):=\int_{0}^{x}e^{-t}t^{a-1}dt$ the lower incomplete Gamma function, we
can write that, for $t>0,$%
\begin{eqnarray*}
P\{T_{\beta } &<&t\}=1-\frac{\xi e^{-\lambda t}}{\mu ^{\alpha }}%
\sum\limits_{j=0}^{\infty }\left( -\frac{\xi e^{-\lambda t}-\mu ^{\alpha }}{%
\mu ^{\alpha }}\right) ^{j}\frac{\gamma (\alpha j+\alpha ;\mu \beta )}{%
\Gamma (\alpha j+\alpha )} \\
&=&1-\frac{\xi \beta ^{\alpha }e^{-\lambda t}}{\mu ^{\alpha }}%
\sum\limits_{j=0}^{\infty }\left( -\frac{\beta ^{\alpha }(\xi e^{-\lambda
t}-\mu ^{\alpha })}{\mu ^{\alpha }}\right) ^{j}E_{1,\alpha j+\alpha +1}(\mu
\beta ).
\end{eqnarray*}%
where we have considered formula (3.7) p.316 in \cite{MIL} together with
formula (4.2.8) in \cite{GOR}.\\
We now compare (in Fig.2) the pdf's of the compound birth process in the two
special cases of Mittag-Leffler and tempered stable addends (given in (\ref%
{tr}) and (\ref{3.27}), respectively) with that of the exponential case:
with respect to the latter, the pdf's fall generally quicker while they have
longer tails. Moreover the presence of the tempering parameter in the
tempered stable case causes both an initial slower fall of the density
compared to the pure $\alpha $-stable case and faster fall for large values
of $y$.

\begin{figure}[]
\subfloat[]{\includegraphics[width =
2in]{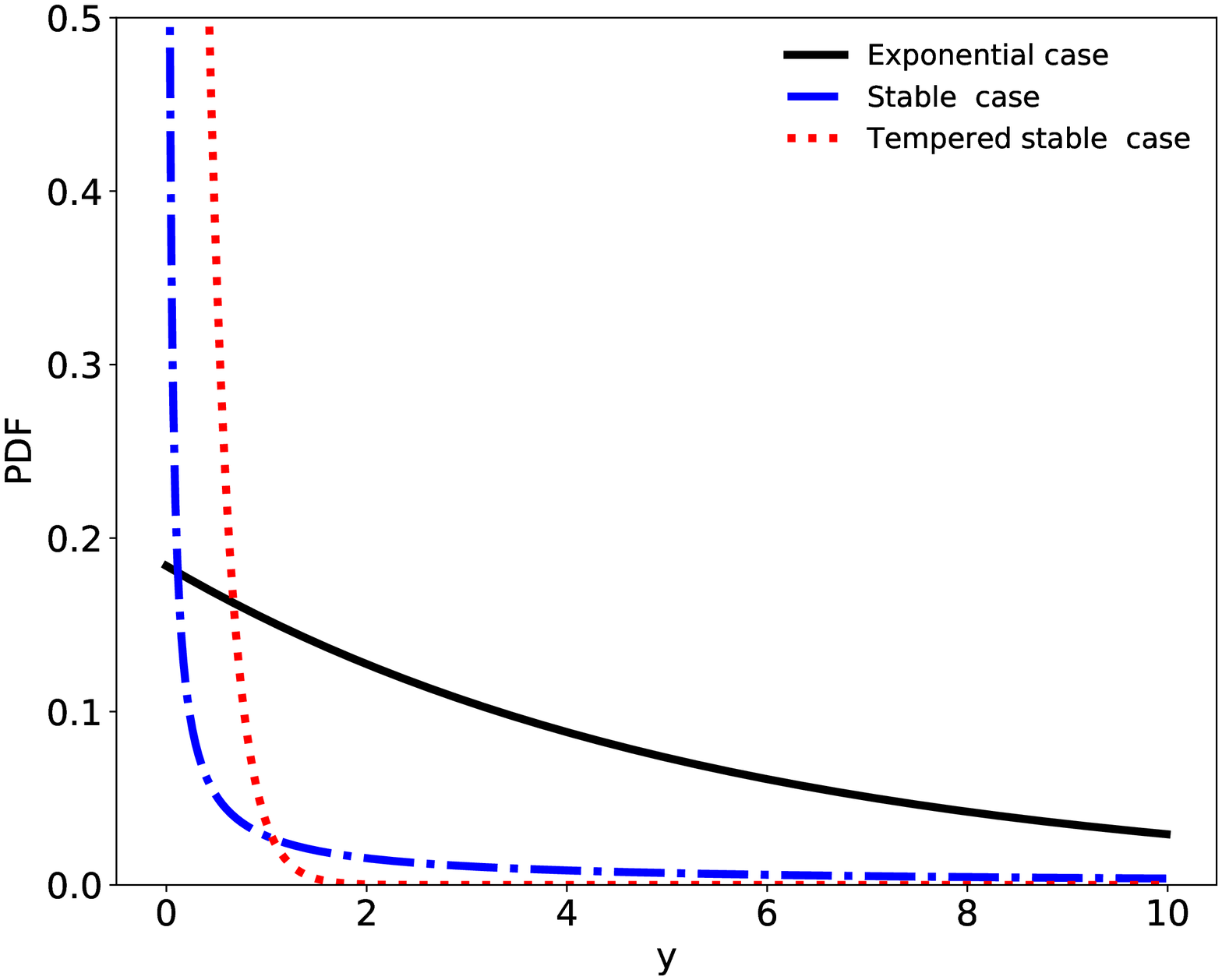}} \subfloat[]{%
\includegraphics[width = 2in]{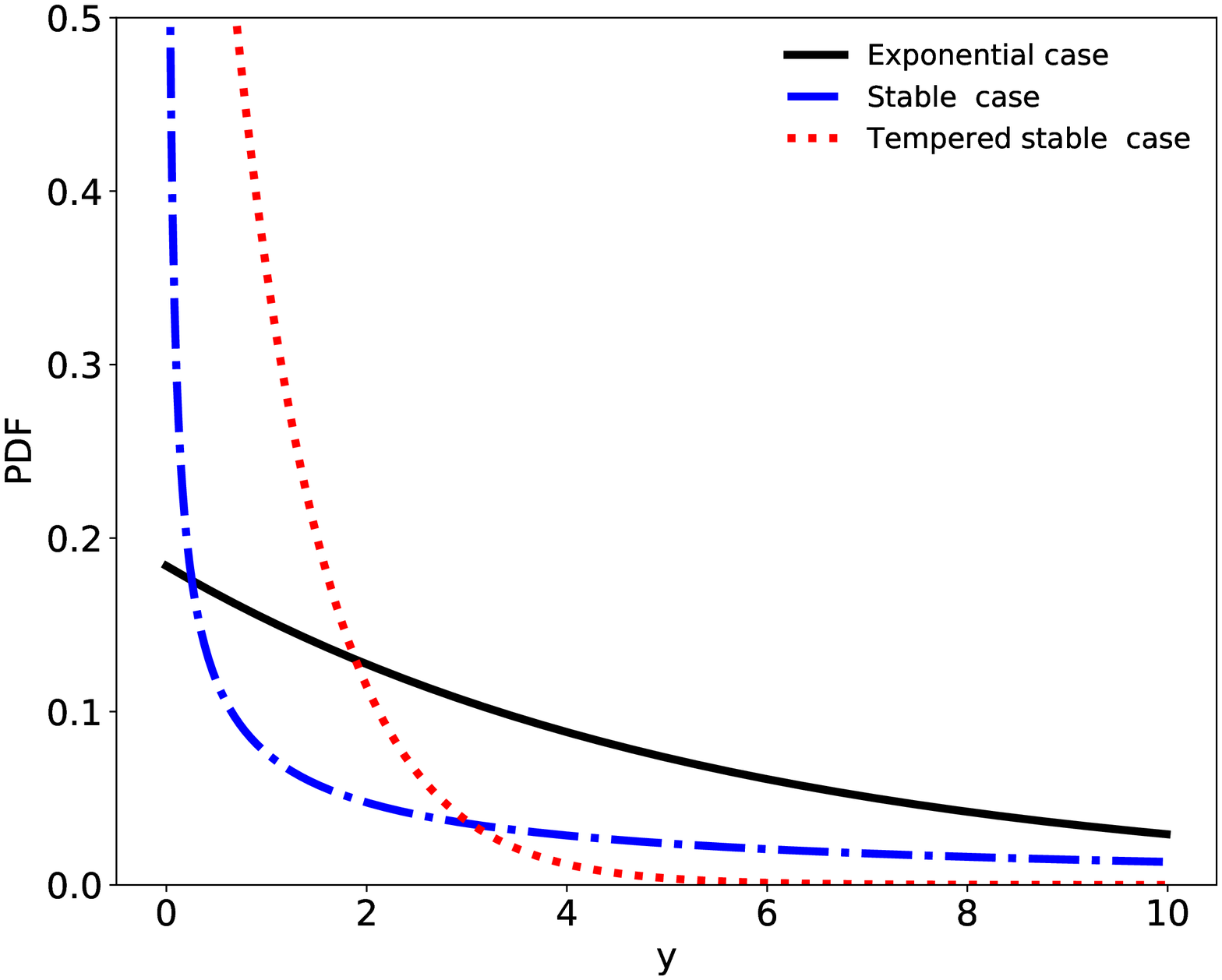}} %
\subfloat[]{\includegraphics[width = 2in]{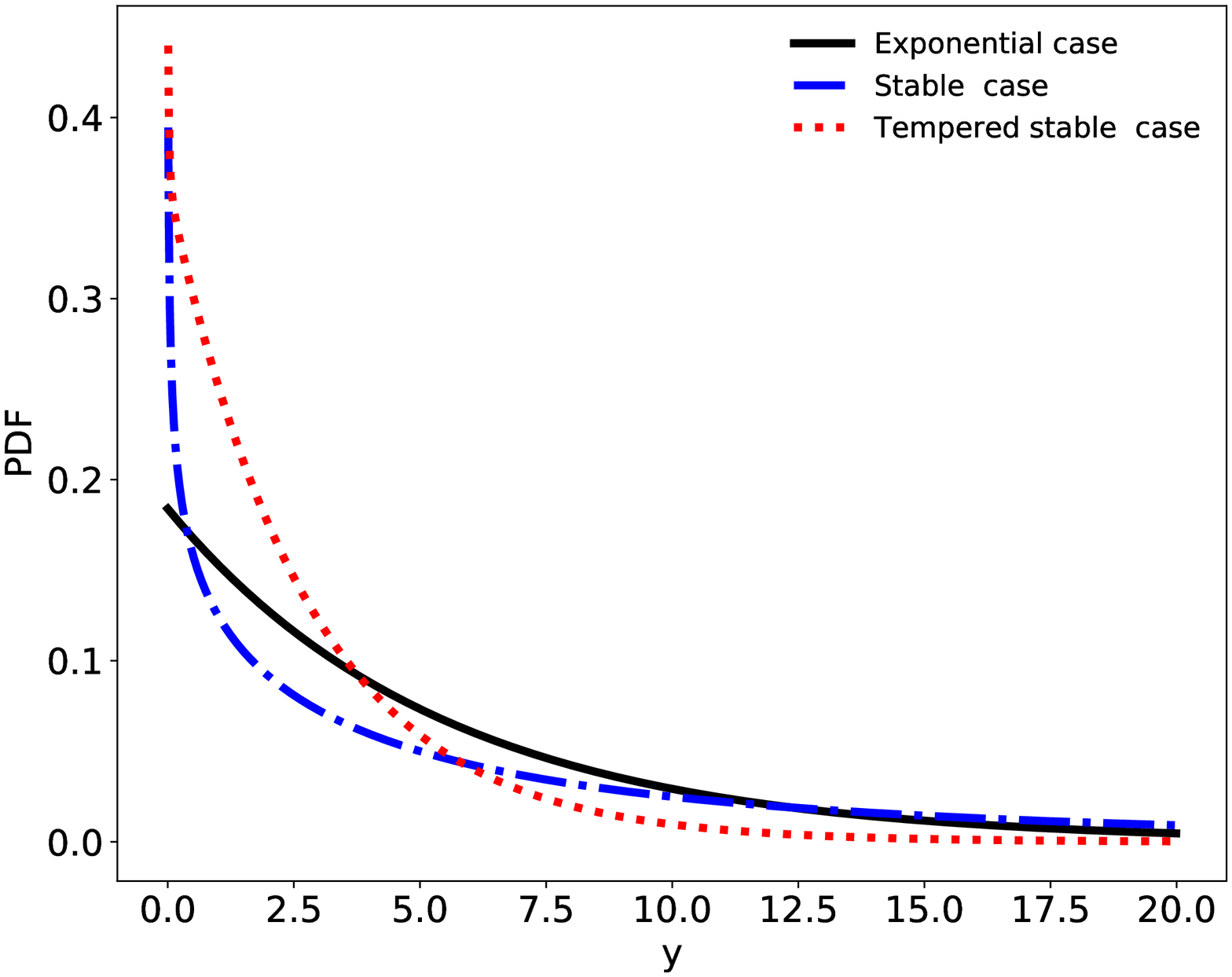}}
\caption{The pdf's of the compound birth process $Y$, with exponential,
Mittag-Leffler and tempered stable addends: for $t=1,\protect\lambda =1,%
\protect\xi =0.5,\protect\alpha =0.2 (a), 0.5 (b), 0.8 (c),\protect\mu =10$.}
\end{figure}

In Fig.3 we explore the influence of the tempering parameter $\mu$ for the
pdf given in (\ref{3.27}), in the tempered case: as expected, the greater
the value of $\mu $ the faster the fall of the tails.
\begin{figure}[]
\centering
\includegraphics[width=.5\textwidth, height =
.3\textwidth]{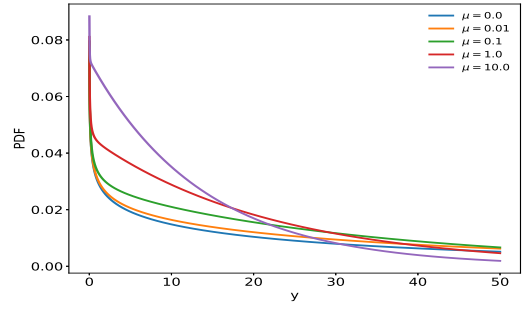}
\caption{The pdf's of the compound birth process $Y_{g}$, with tempered
stable addends, for different values of $\protect\mu $ and for $\protect\xi =%
\protect\lambda =1,\protect\alpha =0.8,t=1$. }
\end{figure}
\\
Finally, we compare the first crossing time probability of $Y$ through the
level $\beta $, in the usual three cases, i.e. exponential, stable and
tempered stable, by plotting it with respect to time (in Fig.\ref{figure4})
and to $\beta $ (in Fig.\ref{figure5}).
\begin{figure}[]
\centering
\includegraphics[width=.5\textwidth, height =
.3\textwidth]{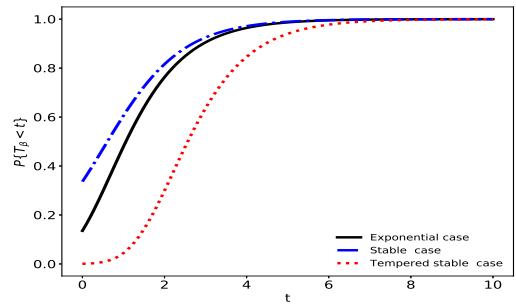}
\caption{The first crossing probability of the level $\protect\beta $, for $%
\protect\mu =5,\protect\xi =\protect\lambda =1,\protect\alpha =0.5,t=1$. }
\label{figure4}
\end{figure}
\\
\begin{figure}[]
\centering
\includegraphics[width=.5\textwidth, height =
.3\textwidth]{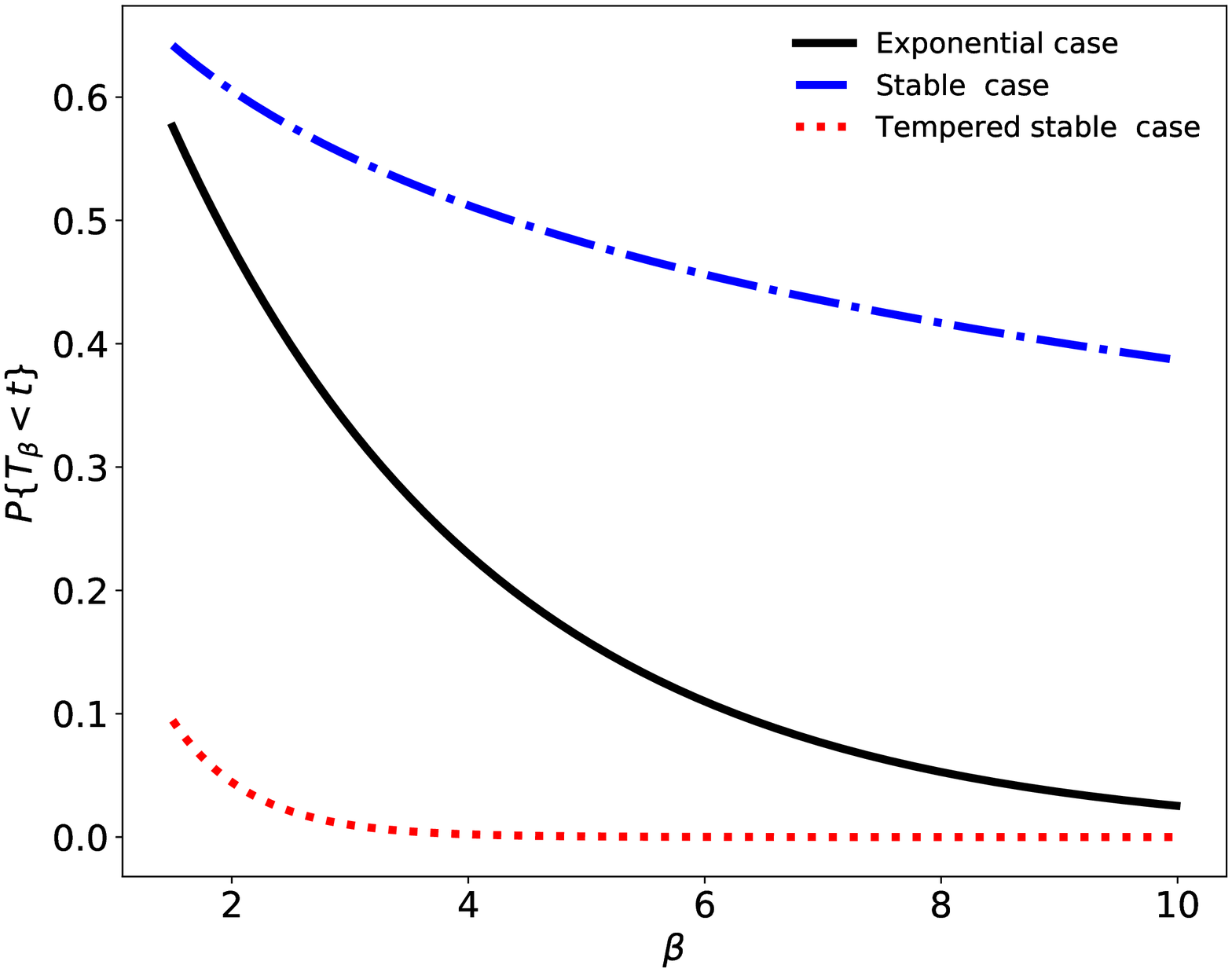}
\caption{The first crossing probability of the level $\protect\beta $=2, for
$\protect\mu =5,\protect\xi =\protect\lambda =1,\protect\alpha =0.5$. }
\label{figure5}
\end{figure}
\\

\textbf{(iii) The gamma case }

For $g(\theta )=\log \left( 1+\frac{\theta }{b}\right) $, for $b>0,$ which
is the Bernstein function associated to the gamma distribution, the law of
the addends $X_{\Gamma }$ can be written as follows:%
\begin{equation*}
f_{X_{\Gamma }}(x)=\xi e^{-bx}\int_{0}^{+\infty }\frac{b^{t}}{\Gamma (t)}%
x^{t-1}e^{-\xi t}dt
\end{equation*}%
with mean value $\mathbb{E}X_{\Gamma }=1/\xi b$ and Laplace transform
\begin{eqnarray}
\widetilde{f}_{X_{\Gamma }}(\theta ) &=&\xi \int_{0}^{+\infty }e^{-bx-\theta
x}\int_{0}^{+\infty }\frac{b^{t}}{\Gamma (t)}x^{t-1}e^{-\xi t}dtdx
\label{inv} \\
&=&\xi \int_{0}^{+\infty }\left( \frac{be^{-\xi }}{b+\theta }\right) ^{t}dt=%
\frac{\xi }{\log (b+\theta )-\log \left( be^{-\xi }\right) }.  \notag
\end{eqnarray}%
In this case, by considering that%
\begin{equation*}
\log \left( 1+\frac{\theta }{b}\right) =\int_{0}^{+\infty }(1-e^{-\theta
x})x^{-1}e^{-bx}dx,
\end{equation*}%
we can write the tail L\'{e}vy measure as%
\begin{equation*}
\nu (ds)=ds\int_{s}^{+\infty }z^{-1}e^{-bz}dz=ds\int_{bs}^{+\infty }\omega
^{-1}e^{-\omega }d\omega =E_{1}(bs)ds,
\end{equation*}%
where $E_{1}(x)=-\func{Ei}(-x)$ and $\func{Ei}(x)$ denotes the exponential
integral. Thus the differential equation governing the process (\ref{bi})
coincides with (\ref{eq}) with convolution-type derivative defined as
follows:%
\begin{equation*}
\mathcal{D}_{t}^{g}u(t):=\int_{0}^{t}\frac{d}{ds}u(t-s)E_{1}(bs)ds.
\end{equation*}%
In this case the Laplace transform of the first-passage time distribution (%
\ref{III}) reduces to%
\begin{equation}
\mathcal{L}\left[ P\{T_{\beta }<t\};\theta \right] =\frac{\log (\theta
+b)-\log b}{\theta \left[ e^{-\lambda t}\xi +\log (\theta +b)-\log b\right] }%
,  \notag
\end{equation}%
which can be inverted by considering (\ref{inv}), as follows, for $t>0$:%
\begin{eqnarray*}
P\{T_{\beta } &<&t\}=\xi e^{-\lambda t}\int_{\beta }^{+\infty
}e^{-by}\int_{0}^{+\infty }\frac{b^{z}}{\Gamma (z)}y^{z-1}\exp \{-\xi
e^{-\lambda t}z\}dzdy \\
&=&\xi e^{-\lambda t}\int_{0}^{+\infty }\frac{\Gamma (z;\beta b)}{\Gamma (z)}%
\exp \{-\xi e^{-\lambda t}z\}dz.
\end{eqnarray*}%
The expected first-passage time through $\beta $ can be written as follows:%
\begin{eqnarray*}
\mathbb{E}T_{\beta } &=&\xi \int_{0}^{+\infty }\frac{\gamma (z;\beta b)}{%
\Gamma (z)}\int_{0}^{+\infty }e^{-\lambda t}\exp \{-\xi e^{-\lambda t}z\}dtdz
\\
&=&\frac{\xi }{\lambda }\int_{0}^{+\infty }\frac{\gamma (z;\beta b)}{\Gamma
(z)}\int_{0}^{1}\exp \{-\xi uz\}dudz \\
&=&\frac{1}{\lambda }\int_{0}^{+\infty }\frac{\gamma (z;\beta b)}{\Gamma
(z+1)}(1-e^{-\xi z})dz.
\end{eqnarray*}%
By applying the monotone convergence theorem and considering that $\frac{%
\gamma (z;\beta b)}{\Gamma (z+1)}\rightarrow \frac{1}{z}$, as $\beta
\rightarrow \infty ,$ it is easy to check that $\mathbb{E}T_{\beta }$ is
infinite in the limit.

\textbf{(iv) The Poisson case }

Let $g(\theta )=\lambda (1-e^{-\theta })$, for $\lambda >0,$ which is the
Bernstein function associated to the Poisson distribution and with L\'{e}vy
measure $\nu (s)=\lambda \delta (s-1)$, where $\delta (\cdot )$ is the
Dirac's delta function. In this case, since the distribution of the addends $%
X^{(\lambda )}$ is discrete and integer valued, we must adapt the notation
of the previous sections: let $p_{x}^{(\lambda )}:=P\{X^{(\lambda )}=x\}$
denote the addends' probability mass function and let $\widetilde{p}%
_{X}^{(\lambda )}(\theta ):=\sum_{x=0}^{\infty }e^{-\theta x}p_{x}^{(\lambda
)}$, then (\ref{lconv}) must be replaced by
\begin{equation*}
\widetilde{F}_{X}^{\ast (n)}(\theta )=\frac{\left[ \widetilde{p}%
_{X}^{(\lambda )}(\theta )\right] ^{n}}{\theta }.
\end{equation*}%
Moreover, formula (\ref{lo}) is substituted by%
\begin{equation*}
p_{x}^{(\lambda )}=\xi \int_{0}^{+\infty }e^{-\xi t}P\{N(t)=x\}dt
\end{equation*}%
Correspondingly, we have, for $Y_{\lambda
}(t)=\sum_{j=1}^{B(t)}X_{j}^{(\lambda )},$ that $q_{x}^{(\lambda
)}(t):=P\{Y_{\lambda }(t)=x\}$ with Laplace transform%
\begin{equation}
\widetilde{q}_{Y}^{(\lambda )}(\theta ):=\sum_{x=0}^{\infty }e^{-\theta
x}q_{x}^{(\lambda )}(t)=\frac{\xi e^{-\lambda t}}{\xi e^{-\lambda
t}+g(\theta )}=\frac{\xi e^{-\lambda t}}{\xi e^{-\lambda t}+\lambda
(1-e^{-\theta })}.  \label{fy2}
\end{equation}

It is easy to see that, in this special case, the addends follow a geometric
distribution of parameter $p=\xi /(\xi +\lambda )$: indeed, here $%
X^{(\lambda )}\overset{d}{=}N(X),$ with $X\sim Exp(\xi )$ independent of $N,$
so that the probability mass function of $X^{(\lambda )}$ can be written as%
\begin{equation}
p_{x}^{(\lambda )}=\xi \int_{0}^{+\infty }e^{-\xi t}P\{N(t)=x\}dt=\frac{\xi
}{\lambda +\xi }\left( \frac{\lambda }{\lambda +\xi }\right) ^{x},\qquad
x=0,1,...
\end{equation}%
The distribution of the process $Y_{\lambda }$ can be written, for $%
y=0,1,... $ and for any $t\geq 0,$ as
\begin{equation}
q_{y}^{(\lambda )}(t)=\frac{\xi e^{-\lambda t}\lambda ^{y}}{y!}%
\int_{0}^{+\infty }e^{-(\xi e^{-\lambda t}+\lambda )z}z^{y}dz=\frac{\xi
e^{-\lambda t}\lambda ^{y}}{\left( \xi e^{-\lambda t}+\lambda \right) ^{y+1}}%
,  \label{nt}
\end{equation}%
which coincides with a geometric probability mass function with parameter $%
p=\xi e^{-\lambda t}/(\xi e^{-\lambda t}+\lambda ).$ This can be checked by
evaluating the Laplace transform of (\ref{nt}), which coincides with (\ref%
{fy2}).

The distribution function of the first-passage time through the level $\beta
$ is given by%
\begin{equation*}
P\{T_{\beta }<t\}=\left\{
\begin{array}{l}
0,\qquad t\leq 0 \\
\left( \frac{\lambda }{\lambda +\xi e^{-\lambda t}}\right) ^{\beta
+1},\qquad t>0%
\end{array}%
\right. ,
\end{equation*}%
which converges to zero, for $\beta \rightarrow +\infty .$

As far as the differential equation satisfied by the survival function of $%
N(X),$ we can specialize the definition (1.10), by considering that $\nu
(ds)=\lambda ds\int_{s}^{+\infty }\delta (x-1)dx=\lambda ds1_{(-\infty
,1]}(s)$ and thus%
\begin{equation*}
\mathcal{D}_{t}^{g}u(t)=\lambda \int_{0}^{t\wedge 1}\frac{d}{ds}%
u(t-s)ds=\lambda \left\{
\begin{array}{l}
u(t)-u(0),\qquad t<1 \\
u(t)-u(t-1),\qquad t\geq 1%
\end{array}%
\right.
\end{equation*}%
Its Laplace transform reads%
\begin{eqnarray}
&&\int_{0}^{+\infty }e^{-\theta t}\lambda \lbrack u(t)-u\left( (t-1)\vee
0\right) ]dt \\
&=&\lambda \widetilde{u}(\theta )-\lambda \int_{0}^{1}e^{-\theta
t}u(0)dt-\lambda \int_{1}^{+\infty }e^{-\theta t}u(t-1)dt  \notag \\
&=&\lambda (1-e^{-\theta })\widetilde{u}(\theta )-\frac{\lambda
(1-e^{-\theta })}{\theta }u(0),  \notag
\end{eqnarray}%
which coincides with (\ref{lapconv}), for this choice of $g.$ As a
consequence, by taking into account Lemma 3.1, we obtain that the survival
function of the addends $P\{X^{(\lambda )}\geq x\}$ satisfies the following
equation (for $x=1,2,..)$%
\begin{equation*}
(\lambda +\xi )u(x)=\lambda u(x-1)
\end{equation*}%
with $u(0)=1,$ as can be checked also directly by considering that $%
P\{X^{(\lambda )}\geq x\}=\left( \frac{\lambda }{\lambda +\xi }\right) ^{x},$
in this case. Theorem 3.2 can be formulated as follows, with the convention
that $q_{-1}(t)=0$: the solution to the initial value problem%
\begin{eqnarray}
\frac{\partial }{\partial t}q_{y}(t) &=&-\lambda ^{2}\frac{e^{\lambda t}}{%
\xi }\left[ \left( q_{y}(t)\ast q_{y}(t)\right) -\left( q_{y-1}(t)\ast
q_{y-1}(t)\right) \right]  \label{eq4} \\
&=&-\lambda ^{2}\frac{e^{\lambda t}}{\xi }[I-\Delta ]\left( q_{y}(t)\ast
q_{y}(t)\right)  \notag
\end{eqnarray}%
(with $t\geq 0,$ $y=0,1,...,$ and $q_{y}(0)=\xi \lambda ^{y}/(\xi +\lambda
)^{y+1}$)$,$ coincides with (\ref{nt}). This can be checked either directly,
by considering that $q_{y}^{(\lambda )}(t)\ast q_{y}^{(\lambda )}(t)=\xi
^{2}e^{-2\lambda t}\lambda ^{y}y/(\xi e^{-\lambda t}+\lambda )^{y+2},$ or by
taking the Laplace transform and verifying that the discrete analogue of (%
\ref{eq3}) holds in this case, i.e.%
\begin{equation}
\frac{\partial }{\partial t}\widetilde{q}_{Y}(\theta )=-\lambda ^{2}\frac{%
e^{\lambda t}}{\xi }(1-e^{-\theta })\left[ \widetilde{q}_{Y}(\theta )\right]
^{2},
\end{equation}%
for any $t\geq 0.$

\

\end{document}